\documentclass[11pt,twoside]{amsart}
\usepackage{url}
\usepackage{amssymb,amsthm,amsfonts,amstext,amsmath}
\usepackage{newcent}       
\usepackage{helvet}         
\usepackage{courier}        
\usepackage{graphicx}
\usepackage{enumerate}\usepackage{mathrsfs}
\usepackage[charter]{mathdesign}
\usepackage[colorlinks=true,linkcolor=blue,citecolor=magenta]{hyperref}
\usepackage{bbm}
\usepackage{a4wide}
\usepackage{geometry}
\usepackage{gnuplottex}
\usepackage{pgfplots}
\usepackage{float}

\usepackage{multicol}
\numberwithin{equation}{section}
\usepackage{color}
\usepackage{float}
\usepackage{mathrsfs}
\usepackage{comment}

\newtheorem{theorem}{Theorem}[section]
\newtheorem{lemma}[theorem]{Lemma}
\newtheorem{proposition}[theorem]{Proposition}

\newtheorem{remark}[theorem]{Remark}

\newcommand{\PP}{\mathbb{P}}

\DeclareMathOperator\erf{erf}

\let\oldtocsection=\tocsection
\let\oldtocsubsection=\tocsubsection
\let\oldtocsubsubsection=\tocsubsubsection
\renewcommand{\tocsection}[2]{\hspace{0em}\oldtocsection{#1}{#2}}
\renewcommand{\tocsubsection}[2]{\hspace{1em}\oldtocsubsection{#1}{#2}}
\renewcommand{\tocsubsubsection}[2]{\hspace{2em}\oldtocsubsubsection{#1}{#2}}
\DeclareRobustCommand{\SkipTocEntry}[5]{}

\usepackage{hyperref}
\hypersetup{
	colorlinks=true,
	linkcolor=blue,
	filecolor=magenta,      
	urlcolor=magenta,
}

\begin{document}
	
\title[Simulating simple random walks with a deck of cards]{Simulating simple random walks with a deck of cards}

\author{R. Alves}
\address[R. Alves]{University of São Paulo, Brazil.}
\email{raphael.alves.duarte@usp.br}

\author{S. Estácio}
\address[S. Estácio]{Federal University of Minas Gerais, Brazil.}
\email{estacio@ufmg.br}

\author{S. Frómeta}
\address[S. Frómeta]{Federal University of Rio Grande do Sul, Brazil.}
\email{susana.frometa@ufrgs.br}

\author{M. Jara}
\address[M. Jara]{Instituto de Matemática Pura e Aplicada, Brazil.}
\email{mjara@impa.br}

\author{R. Marinho}
\address[R. Marinho]{Federal University of Santa Maria, Brazil.}
\email{rodrigo.marinho@ufsm.br}
\urladdr{\url{https://marinhor.weebly.com}} 

\author{L. F. S. Marques}
\address[L. F. S. Marques]{University of São Paulo, Brazil; Université Paris Dauphine--PSL, France.}
\email{lfsmarquess@gmail.com}

\author{J. V. A. Pimenta}
\address[João V. A. Pimenta]{University of São Paulo, Brazil.}
\email{joaovictorpimenta@usp.br}

\begin{abstract}
When we want to simulate the realization of a symmetric simple random walk on $\mathbb Z^d$, we use $(2d)$-side fair dice to decide to which neighbor it jumps at each step if $d\geq 2$ or we simply use a fair coin when $d=1$. Assume that instead of using a dice or a coin we want to do a simulation using a well shuffled deck with $K$ cards of each of the $2d$ suits. In the first step the probability of jumping to each neighbor is $(2d)^{-1}$, but from the second step it becomes biased. Of course if we continue performing this simulation, the total variation distance between its law and the law of the random walk will increase until all cards are used. In this paper we investigate the minimum number of cards $N=2d K$ that a deck must contain so that the total variation distance between the law of a $n$-step simulation and the law of a $n$-step realization of the random walk is smaller than a chosen threshold $\varepsilon \in (0,1)$. More generally, we prove that when $N=cn$ this distance converges, as $n \to \infty$, to a Gaussian profile which depends on $c\geq 2d$. Furthermore, our analysis shows that this Gaussian profile vanishes as $c \to \infty$, proving the convergence of a multivariate hypergeometric distribution to a multinomial distribution in total variation.
\end{abstract}


\maketitle

\tableofcontents

\section{Introduction}

The discrete time symmetric simple random walk (SSRW) on $\mathbb Z^d$ is one of the simplest examples of discrete time Markov chains: when at a vertex $x \in \mathbb Z^d$ it jumps to one of its $(2d)$-equally-likely neighbors. Despite its simplicity, there are several studies involving random walks in biology, financial modeling, statistical physics and other sciences.

It is very common to use the SSRW on $\mathbb Z$, starting at the origin, to illustrate the Law of Large Numbers (LLN) and the Central Limit Theorem (CLT). The idea is that the position of the SSRW after $n$ steps can be seen as the sum of $n$ i.i.d. random variables taking values in $\{-1,1\}$, and therefore, the average of this random variables converges almost surely to zero and the distribution of the aforementioned sum converges to a Gaussian distribution. Since we can associate i.i.d. random variables taking values in $\{-1,1\}$ with independent fair coin tosses, this example can be easily explained to college students through a practical experiment, as well as Buffon's Needle Experiment, see~\cite{buffon}. For the SSRW on $\mathbb Z^d$ one can adapt this experiment using $d$ fair coins at each step, which would not be very practical if $d$ is too large, or a $(2d)$-sided dice.

In this work we are interested in a practical experiment to simulate the SSWR on $\mathbb Z^d$, but using a well shuffled deck of cards instead of coins or dices. If $d=1$ we can color half of the cards in black and the other half in red and, after mixing the deck, we remove the cards -- one by one -- and update the process according to the color of the card. More precisely, if the card is black then the process jumps to the right, otherwise it jumps to the left. When the first card is removed, the probability of jumping to each sense is $1/2$, but from the second step it becomes biased. Thus, it seems that if we continue performing this simulation, its law and the law of the SSRW will become more distant. The same happens for $d\geq 2$, dividing the cards into $2d$ suits of the same size. Our goal is to determine the minimum number of cards $N$ that a deck, divided into $2d$ suits, must contain so that the total variation distance between the law of a $n$-step simulation and the law of a $n$-step realization of the SSRW is smaller than a chosen threshold $\varepsilon \in (0,1)$. More generally, we prove that when $N=cn$ this distance converges, as $n \to \infty$, to a Gaussian profile which depends on $c\geq 2d$, see Theorem~\ref{main}.

The Gaussian shape of the profile derived in this paper is actually expected. Indeed, the analogy with coins made for the SSWR on $\mathbb Z$ allows us to associate the position of the process after $n$ steps with a Binomial random variable. Similarly, we can associate the position of the $d$-dimensional process after $n$ steps with a Multinomial random variable. Therefore, the CLT assures that, independently of the dimension, the SSRW approaches a Gaussian distribution. Furthermore, since there are $N$ cards in the deck and exactly $N(2d)^{-1}$ of the cards correspond to each suit, the position of the simulation of interest depends only on the numbers of cards of each suit among the $n$ removed cards. Therefore, we can associate it with a Multivariate Hypergeometric random variable, which by the CLT also approaches a Gaussian distribution. Given this explanation, the scaling $N=cn$ is the one for which we can see the convergence of the discrete random variables to the Gaussian ones. Moreover, as can be seen in Theorem~\ref{main}, the Gaussian profile converges to zero as $c \to \infty$, showing the convergence, in total variation, of the Multivariate Hypergeometric distribution to the Multinomial distribution. Although well known, the proof of this classical limit could not be found in the literature for $d\geq 2$, not even for weaker types of convergence.

The paper is organized as follows: in Section~\ref{definitions} we rigorously define the SSRW on $\mathbb Z^d$ and its simulation. We also introduce some notation and we enunciate our main result, Theorem~\ref{main}. In Section~\ref{discussion} we discuss the simulation when using only the colors of the cards ($d=1$) and when using only four suits ($d=2$), explaining, algorithmically, how to proceed in higher dimensions. In each particular case we explain how far one can go in a simulation with a fifty-two-card deck. Last, but not least, we prove Theorem~\ref{main} in Section~\ref{proof}.

\section{Definitions and main result}\label{definitions}

Let $(e_i)_{i =1}^d$ denote the canonical basis of $\mathbb R^d$. Let $(\xi_j)_{j \in \mathbb N}$ be an independent and identically distributed sequence of random variables with $\mathbb P(\xi_j=e_i)=\mathbb P(\xi_j=-e_i)=\frac{1}{2d}$ for all $i \in \{1,\ldots,d\}$. A symmetric simple random walk (SSRW) on $\mathbb Z^d$ can be defined as a sequence $(X_k)_{k\geq 0}$ where $X_0$ is the origin and $X_k=\sum_{j=1}^k \xi_j.$

Denote the euclidean norm of a vector $v=(v_1,\ldots,v_{m}) \in \mathbb R^{m}$ by
\begin{flalign}\label{euclidean}
	\| v \|=\left(\sum_{i=1}^{m} |v_i|^2\right)^{1/2}
\end{flalign}
for any $m \in \mathbb N$. Notice that, for a fixed $n\in\mathbb N$, $(X_k)_{0\leq k \leq n}$ takes values in the set 
\begin{equation*}
	\Gamma_n=\left\{x=(x_0,\ldots,x_n) \in \mathbb Z^{d(n+1)}: x_0=(0,\ldots,0) \text{ and} \ \|x_j-x_{j-1}\|=1 \ \text{for all} \ j =1,\ldots,n\right\}.
\end{equation*} 
By the independence of the random variables $(\xi_j)_{1\leq j \leq n}$, we have 
\begin{equation*}
	\mathbb P((X_k)_{0\leq k\leq n}= x) = \frac{1}{|\Gamma_n|} = \frac{1}{(2d)^n}
\end{equation*} 
for any $x \in \Gamma_n$. That is, $(X_k)_{0\leq k \leq n}$ is uniformly distributed in $\Gamma_n$. Also notice that there exists a bijection between $\Gamma_n$ and 
\begin{equation*}
	\Omega_n := \{0,1,\ldots,2d-1\}^n.
\end{equation*}

Let us consider, for $N\in\mathbb N$, a deck of $N$ cards defined by the set $\Lambda_N=\{1,\ldots, N\}$ where each $j\in\Lambda_N$ represents a card. We will assume that $N=2d K$ for $K \in \mathbb N$. Therefore we can suppose that each card is labeled not only with a number but also with one of $2d$ suits, and that there are exactly $K$ cards of each suit. We denote by $S_N$ the set of permutations of $\Lambda_N$, that is
\begin{equation*}
	S_N=\{(\sigma(1),\ldots, \sigma(N)):\sigma:\Lambda_N\to\Lambda_N \text{~is a bijection}\}.    
\end{equation*}

Let us consider the suit function $\varphi:\Lambda_N\to \mathbb Z/ (2d\,\mathbb Z)$ defined by
\begin{equation*}
	\varphi(j) = r \,\,\,\,\, \text{ if } j \equiv r \text{ (mod }2d\text{)}.
\end{equation*}
For $n<N$, the function $X:S_N\to\Omega_n$, defined by
\begin{equation*}
	X(\sigma)=(\varphi(\sigma_1),\ldots,\varphi(\sigma_n)),
\end{equation*}
represents a simulation of the random walk with size $n$ by using a deck of $N$ cards.

Let us consider $\mathbb Q_N$ and $\mathbb P_n$ uniform distributions in $(S_N,\mathcal P(S_N))$ and $(\Omega_n,\mathcal P(\Omega_n))$, respectively. Each function $X:S_N\to\Omega_n$ inducts a probability measure $\mu_X$ in $(\Omega_n,\mathcal P(\Omega_n))$ given by
\begin{equation*}
	\mu_X(\omega)=\mathbb Q_N(X=\omega), ~~~~\omega\in\Omega_n.
\end{equation*}
We are interested in quantifying how close the simulation defined by $X$ is to the SSRW on $\mathbb Z^d$. For that, we use the total variation distance. That is,  we would like to estimate the total variation distance between the probabilities $\mu_X$ and $\mathbb P_n$, which is given by \begin{equation}\label{dvt1}
	d_n(N)=\sum_{\omega \in \Omega_n}\left[ \mu_X(\omega) - \PP_n(\omega)\right]^{+}=\sum_{\omega \in \Omega_n}\left[ \mu_X(\omega) - (2d)^{-n}\right]^{+}.
\end{equation} 
Above we use the notation $[f(\omega)]^+=f(\omega)\cdot\mathbb 1_{\{ f(\omega)> 0 \}}$.

In order to obtain an expression for $\mu_X(\omega)$ and, consequently, for $d_n(N)$ we need to recall the definition of multinomial coefficients: for each $\ell,m \in \mathbb N$ define
\begin{flalign*}
	\Pi_{\ell}^{m}:=\left\{ \lambda=(\lambda_1,\ldots,\lambda_\ell) \in \mathbb R^\ell ;\,\lambda_i \in \mathbb N \cup \{0\} \text{ for all } i \in \{1,2,\ldots,\ell\} \text{ and } \sum_{i=1}^\ell \lambda_i = m    \right\}.
\end{flalign*}
The multinomial coefficient is defined as
\begin{flalign*}
	\binom{m}{\lambda_1,\ldots,\lambda_\ell}= \frac{m!}{\prod_{i=1}^\ell \lambda_i!}
\end{flalign*}
for each $\lambda=(\lambda_1,\ldots,\lambda_\ell) \in \Pi_\ell^m$.

For each $i \in \mathbb Z/ (2d\,\mathbb Z)$, let $\lambda_\omega^i$ be cardinality of $\{j \in \{1,\ldots,n\}: \omega_j=~i\}$. We can compute 
\begin{equation*}
	\mu_X(\omega) = \frac{\binom{N-n}{K-\lambda_\omega^1,\ldots,K-\lambda_\omega^{2d}}}{\binom{N}{K,\ldots, K}},
\end{equation*} 
since the denominator counts the equiprobable elements of $\Omega_N$ with exactly $K$ terms equals to each element of $\mathbb Z/ (2d\,\mathbb Z)$, and the numerator counts all the sequences whose first $n$ terms coincides with those of $\omega$. 

For each $\omega \in \Omega_n$ define $\lambda_\omega=(\lambda_\omega^0,\ldots,\lambda_\omega^{2d-1})$. Notice that $\mu_X(\omega)=\mu_X(\omega')$ whenever  $\lambda_\omega=\lambda_{\omega'}$. 
Then, we can rewrite equation~\eqref{dvt1} as
\begin{flalign}
	d_n(N)=&\sum_{\lambda_1=0}^n \sum_{\lambda_2=0}^{n-\lambda_1}\cdots \sum_{\lambda_{2d}=0}^{n-\sum_{i=1}^{2d-1} \lambda_i}
	\left[ \frac{\binom{N-n}{K-\lambda_1,\ldots,K-\lambda_{2d}}\binom{ n}{\lambda_1, \ldots, \lambda_{2d}}}{\binom{N}{K,\ldots, K}} - \frac{1}{(2d)^n} \binom{ n}{\lambda_1, \ldots, \lambda_{2d}} \right]^+. \label{dvt2}
\end{flalign}

The above expression shows that 
$d_n(N)$ actually is the the total variation distance between a multivariate hypergeometric $H_n\sim H(N,K_1=K,\ldots,K_{2d}=K,n)$ and a multinomial distribution $B_n\sim B(n,p_1=1/2,\ldots,p_{2d}=1/2)$.  
It is well known that for $n$ fixed, $d_n(N)\to 0$ as $N\to\infty$ (see the one-dimensional case in~\cite{clt}). However, we are interested in a finer estimate that provides also information about the behavior of $d_n(N)$ for $n$ and $N$ large.

Notice that $d_n(N) \in [0,1]$. We are interested in finding the smallest number of cards $N$ that a deck must contain so that a simulation with $n$ iterations presents $d_n(N)$ smaller than a chosen threshold $\varepsilon \in (0,1)$. This problem, which has similarities to the problem of determining mixing times of Markov chains (see~\cite{peres} for an introduction to the subject), is equivalent to determining the largest number $n$ of iterations that one can perform on a simulation with a $N$-card deck so that $d_n(N)$ is smaller than $\varepsilon$.

From now on we use the notation $a_n=\mathcal O(b_n)$ if there exists a positive constant $C$ such that $|a_n| \leq C\,|b_n|$ for sufficiently large $n$. In this case we say that $a_n$ is of order $b_n$.
Our main result states that, for $N=cn$, the total variation distance is of the order of a constant that depends on $c$, which is explicitly given, plus an error term of order $n^{-1/2}$.

\begin{theorem}\label{main}
	Let $c \geq 2d$. Let $r=r(c)=\sqrt{(4d^2-2d)(c-1)\log(\frac{c}{c-1})}$. Let $B_r^m=\left\{  X \in \mathbb R^{m}; \| X \| \leq r  \right\}$ be the $m$-ball that is centered at the origin and whose radius is equal to $r$. Let $\varPi_{m}=\{ x=(x_1,\ldots,x_{m}) \in \mathbb R^{m} ; \sum_{i=1}^{m} x_i=0\}$ be the hyperplane that is orthogonal to the vector $(1,\ldots,1)$ and contains the origin. Let $\mathcal G: B_r^{2d} \cap \varPi_{2d} \to [0,1]$ be the Gaussian profile defined as
	\begin{equation*}
		\mathcal G(X)=\left( \pi^{1/2-d} \, 2^{3/2-2d}\,d^{1-d} \right) \left[\left(\frac{c}{c-1}\right)^{d-1/2}\exp{\left(-\frac{c\|X\|^2}{4d(c-1)}\right)}-\exp{\left(-\frac{\|X\|^2}{4d}\right)}\right]
	\end{equation*}
	for every $X \in B_r^{2d} \cap \varPi_{2d}$. Then,
	\begin{flalign*}
		d_n(cn)&=\int_{B_r^{2d} \cap \varPi_{2d}}\mathcal G(X)\, dX+\mathcal{O}\left(\frac{1}{\sqrt{n}}\right).
	\end{flalign*}
\end{theorem}	

\section{Discussion of the main result}~\label{discussion}

In this section we present some elementary computations that lead us to a simpler formulation of Theorem \ref{main} for the particular, and most natural, cases $d=1$ and $d=2$.

\subsection{Using only the colors of cards}

When $d=1$, the SSRW on $\mathbb Z$ jumps to the left or to the right according to the color of the card which was removed from the deck. In this case, $B_r^{2} \cap \varPi_{2d}$ is simply the interval $[-r/\sqrt{2},r/\sqrt{2}]$. Therefore,
\begin{flalign*}
	\int_{B_r^{2} \cap \varPi_{2d}}\mathcal G(X)\, dX &=
	\int_{-r/\sqrt{2}}^{r/\sqrt{2}} \frac{1}{\sqrt{2 \pi}} \left[\sqrt{\frac{c}{c-1}}\exp{\left(-\frac{c\,\|(x,-x)\|^2}{4(c-1)}\right)}-\exp{\left(-\frac{\|(x,-x)\|^2}{4}\right)}\right] dx\\
	&=\int_{-r/\sqrt{2}}^{r/\sqrt{2}}\frac{1}{\sqrt{2\pi}}\left[\sqrt{\frac{c}{c-1}}\exp{\left(-\frac{cx^2}{2(c-1)}\right)}-\exp{\left(-\frac{x^2}{2}\right)}\right]dx.
\end{flalign*}
Taking $\xi=x\sqrt{c/(2(c-1))}$ and $\eta=x/\sqrt{2}$, we obtain
\begin{flalign*}
	\int_{B_r^{2} \cap \varPi_{2d}}\mathcal G(X)\, dX &=\frac{1}{\sqrt{\pi}}  \int_{-r\sqrt{c/(4(c-1))}}^{r\sqrt{c/(4(c-1))}} e^{-\xi^2} d\xi - \frac{1}{\sqrt{\pi}}  \int_{-r/2}^{r/2} e^{-\eta^2} d\eta.
\end{flalign*}

Since in this case $r=\sqrt{2(c-1)\log(\frac{c}{c-1})}$, when $d=1$ Theorem~\ref{main} can be enunciated as follows:
\begin{theorem}\label{theo1}
	If $d=1$ then for all $c\geq 2$,
	\begin{equation*}
		d_n(cn)=\erf\left(\sqrt{\frac{ c\log\left({c}/{c-1}\right)}{2}}\right) - \erf\left(\sqrt{\frac{(c-1)\log\left({c}/{c-1}\right)}{2}}\right) + \mathcal{O}\left(\frac{1}{\sqrt{n}}\right),
	\end{equation*}
	where $\erf(x)=\frac{1}{\sqrt{\pi}}\int_{-x}^x e^{-t^2}dt$ is the Gauss error function. 
\end{theorem}
Figure~\ref{graph1} shows the graph of the function given in Theorem~\ref{theo1} which is the asymptotic profile of $d_n(cn)$ when $d=1$.
\begin{figure}[H]
	\centering
	\includegraphics[scale=0.35]{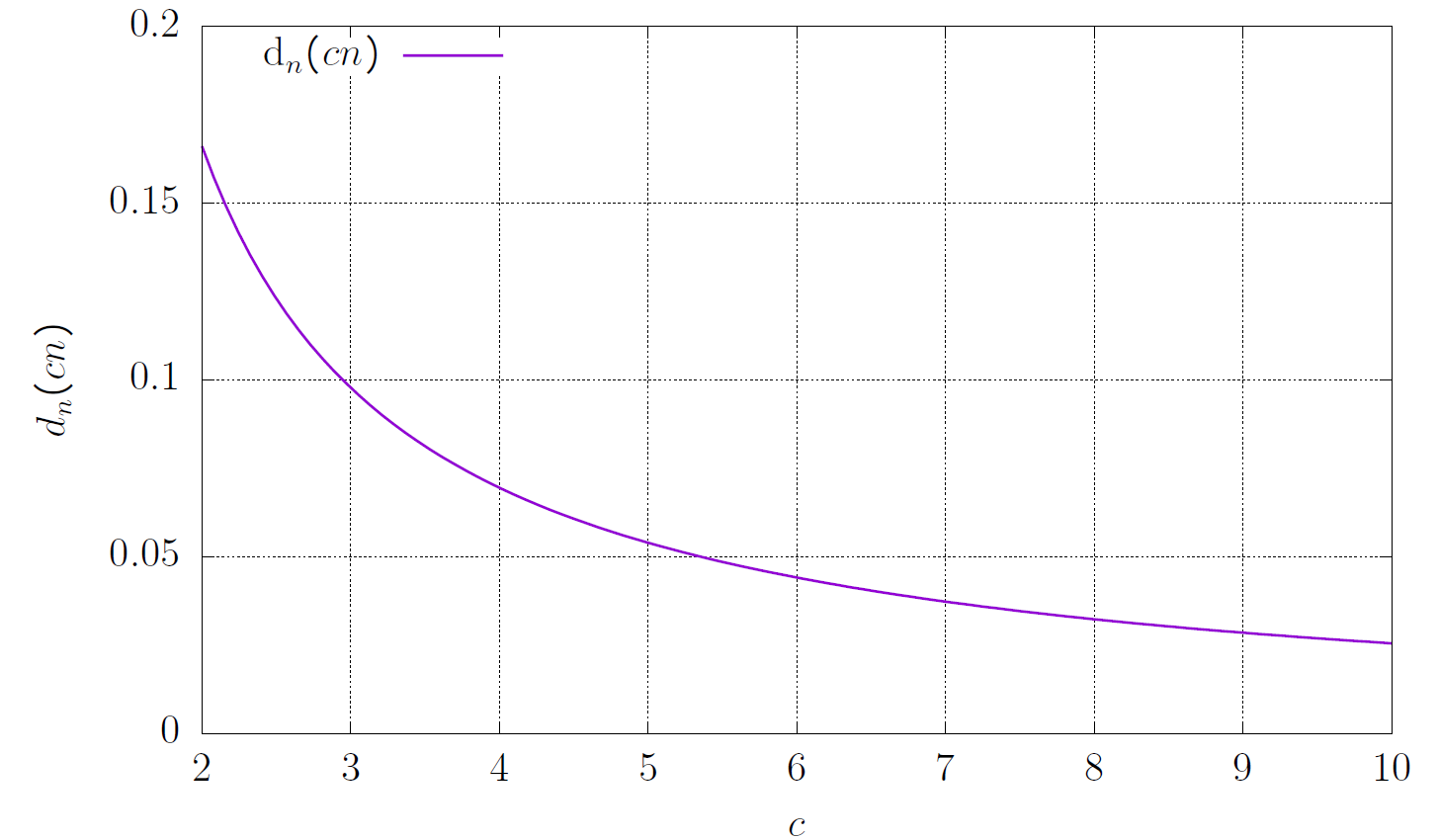}
	\caption{Graph of the asymptotic profile of $d_n(cn)$ when $d=1$.}
	\label{graph1}
\end{figure}

Table~\ref{table1} informs us that with a $52$-card deck one can simulate $26$-steps of the SSRW on $\mathbb Z$ and $d_{26}(52)$ will be close to $0.0160$. Similarly, $d_{17}(52)$ and $d_{9}(52)$ (when using $17$ and $9$ cards, respectively) are close to $0.100$ and $0.050$, respectively.
\begin{table}[H]
	\centering
	\begin{tabular}{|c|c|}
		\hline
		$c$    & $d_n(cn)$ \\ \hline
		$2.00$   & $0.160$     \\ \hline
		$2.94$   & $0.100$     \\ \hline
		$5.35$   & $0.050$     \\ \hline
		$24.70$  & $0.010$     \\ \hline
		$48.89$  & $0.005$     \\ \hline
		$242.47$ & $0.001$     \\ \hline
	\end{tabular}
	\caption{Values of $d_n(cn)$ for given values of $c$. The value $c=2$ is the smallest possible choice; the other values were chosen so that $d_n(cn)$ assumed values which are often chosen.}
	\label{table1}
\end{table}

\subsection{Using a deck with four suits}

When $d=2$, the SSRW on $\mathbb Z^2$ jumps to one of the $2d$ senses according to the suit of the card which was removed from the deck. In this case, $B_r^{4} \cap \varPi_{4}=S \,\cup \text{ int }S$ where $S$ is the spheroid $S=\left\{ x^2+y^2+z^2+xy+xz+yz = r^2/2 \right\}$ and $\text{ int }S$ its interior, that is,
\begin{flalign*}
	B_r^{4} \cap \varPi_4 = \left\{ x^2+y^2+z^2+xy+xz+yz \leq r^2/2 \right\}=:W_{xyz}.
\end{flalign*}
Therefore,
\begin{equation}
	\begin{split}
		\int_{B_r^4 \cap \varPi_4}\mathcal G(X)\, dX =\int\int \int_{W_{xyz}} \frac{1}{\sqrt{128 \pi^3}} \left[\left(\frac{c}{c-1}\right)^{3/2}\exp{\left(-\frac{c\,\|(x,,y,z,-x-y-z)\|^2}{8(c-1)}\right)} \right. \\
		\left. -\exp{\left(-\frac{\|(x,,y,z,-x-y-z)\|^2}{8}\right)}\right] dx dy dz \\
		=\int\int \int_{W_{xyz}} \frac{1}{\sqrt{128 \pi^3}} \left[\left(\frac{c}{c-1}\right)^{3/2}\exp{\left(-\frac{c\, (x^2+y^2+z^2+xy+xz+yz)}{4(c-1)}\right)}\right. \\
		\left. -\exp{\left(-\frac{(x^2+y^2+z^2+xy+xz+yz)}{4}\right)}\right] dx dy dz .\label{g2}
	\end{split}
\end{equation}

We can find a change of variables $\varphi(u,v,w)=(x,y,z)$ that maps $W_{xyz}$ onto the unit ball $B_1^3$ so that we can apply spherical coordinates in the above integral. Indeed, we follow the algorithm explained in the proof of Sylvester's Theorem in~\cite{bueno}. Let $q: \mathbb R^3 \to \mathbb R$ be the quadratic form defined by
\begin{flalign*}
	q(x,y,z)=x^2+y^2+z^2+xy+xz+yz
\end{flalign*}
for any $x \in \mathbb R^3$. Isolating the variable $x$, we obtain
\begin{flalign*}
	q(x,y,z)&=x^2+(y+z)x+y^2+z^2+yz=\left( x +\frac{y+z}{2}  \right)^2 + \frac{3}{4} y^2+ \frac{3}{4} z^2 + \frac{1}{2} yz.
\end{flalign*}
Taking $s=x+\frac{y+z}{2}$ and isolating the variable $y$, we obtain
\begin{flalign*}
	q(x,y,z)&=s^2 + \frac{3}{4}\left(  y + \frac{z}{3} \right)^2 + \frac{2}{3}z^2.
\end{flalign*}
Taking $t=y+\frac{z}{3}$ we obtain
\begin{flalign*}
	q(x,y,z)&=s^2 + \frac{3}{4}t^2 + \frac{2}{3}z^2.
\end{flalign*}
Therefore, the change of variables
\begin{flalign}\label{system}
	\begin{cases}
		u=\frac{\sqrt{2}}{r}s=\frac{\sqrt{2}}{r} x + \frac{\sqrt{2}}{2r} y+ \frac{\sqrt{2}}{2r} z\\
		v= \frac{\sqrt{3}}{r\sqrt{2}} t =  \frac{\sqrt{3}}{r\sqrt{2} } y+ \frac{\sqrt{3}}{3 r\sqrt{2}} z\\
		w=\frac{2 }{r \sqrt{3}} z
	\end{cases}
\end{flalign}
implies that
\begin{flalign*}
	q(x,y,z)=\frac{r^2}{2} \left( u^2+v^2+w^2  \right)
\end{flalign*}
for any $(x,y,z) \in \mathbb R^3$ and maps the spheroid $W_{xyz}$ onto $B_1^3$. Solving the linear system (\ref{system}), we obtain that
\begin{flalign*}
	\begin{cases}
		x= \frac{\sqrt{2}}{2} r u - \frac{\sqrt{6}}{6} r v - \frac{\sqrt{3}}{6} r w\\
		y = \frac{\sqrt{6}}{3} r v -\frac{\sqrt{3}}{6} r w\\
		z= \frac{\sqrt{3}}{2} r w.
	\end{cases}
\end{flalign*}
Therefore, we have
\begin{flalign*}
	\begin{vmatrix}
		\frac{\partial x}{\partial u} & \frac{\partial x}{\partial v} & \frac{\partial x}{\partial w}  \\
		\frac{\partial y}{\partial u} & \frac{\partial y}{\partial v} & \frac{\partial y}{\partial w}  \\
		\frac{\partial z}{\partial u} & \frac{\partial z}{\partial v} & \frac{\partial z}{\partial w}
	\end{vmatrix}
	&=
	\begin{vmatrix}
		\frac{\sqrt{2}}{2} r & -\frac{\sqrt{6}}{6} r &- \frac{\sqrt{3}}{6} r \\
		0 & \frac{\sqrt{6}}{3} r & -\frac{\sqrt{3}}{6} r \\
		0 & 0 & \frac{\sqrt{3}}{2} r
	\end{vmatrix}
	=\frac{r^3}{2}.
\end{flalign*}
Hence, we can rewrite identity (\ref{g2}) as
\begin{equation}
	\begin{split}
		\int_{B_r^4 \cap \varPi_4}\mathcal G(X)\, dX =\int\int \int_{B_1^3} \frac{r^3}{16\sqrt{2 \pi^3}} \left[\left(\frac{c}{c-1}\right)^{3/2}\exp{\left(-\frac{cr^2\, (u^2+v^2+w^2)}{8(c-1)}\right)}\right. \\
		\left. -\exp{\left(-\frac{r^2(u^2+v^2+w^2)}{8}\right)}\right] du dv dw . \nonumber
	\end{split}
\end{equation}
Now, changing the variables $u,v,w$ in the above integral to the spherical coordinates
\begin{flalign*}
	\begin{cases}
		u= \rho \sin \phi \cos \theta\\
		v=\rho \sin \phi \sin \theta\\
		w=\rho \cos  \phi
	\end{cases}
\end{flalign*}
with $\rho \in [0,1]$, $\phi \in [0,\pi]$ and $\theta \in 2\pi$, we obtain
\begin{equation}
	\begin{split}
		\int_{B_r^4 \cap \varPi_4}\mathcal G(X)\, dX =\int_0^{2\pi}\int_0^{\pi} \int_0^1 \frac{r^3 \rho^2 \sin \phi}{16\sqrt{2 \pi^3}} \left[\left(\frac{c}{c-1}\right)^{3/2}\exp{\left(-\frac{cr^2\, \rho^2}{8(c-1)}\right)}\right. \\
		\left. -\exp{\left(-\frac{r^2\rho^2}{8}\right)}\right] d\rho d\phi d\theta . \nonumber
	\end{split}
\end{equation}
By Fubini's Theorem we obtain
\begin{flalign*}
	\int_{B_r^4 \cap \varPi_4}&\mathcal G(X)\, dX =\int_0^{\pi} \int_0^1 \frac{ r^3 \rho^2 \sin \phi}{8\sqrt{2 \pi}} \left[\left(\frac{c}{c-1}\right)^{3/2}\exp{\left(-\frac{cr^2\, \rho^2}{8(c-1)}\right)} -\exp{\left(-\frac{r^2\rho^2}{8}\right)}\right] d\rho d\phi\\
	&= \frac{ r^3 }{8\sqrt{2 \pi}}\left( \int_0^\pi \sin \phi\, d\phi\right) \int_0^1 \rho^2 \left[\left(\frac{c}{c-1}\right)^{3/2}\exp{\left(-\frac{cr^2\, \rho^2}{8(c-1)}\right)} -\exp{\left(-\frac{r^2\rho^2}{8}\right)}\right] d\rho\\
	&=\frac{ r^3 }{4\sqrt{2 \pi}}\int_0^1 \rho^2 \left[\left(\frac{c}{c-1}\right)^{3/2}\exp{\left(-\frac{cr^2\, \rho^2}{8(c-1)}\right)} -\exp{\left(-\frac{r^2\rho^2}{8}\right)}\right] d\rho. 
\end{flalign*}
Finally, taking $\xi=r\rho\,\sqrt{c/(8(c-1))}$ and $\eta=r\rho/\sqrt{8}$ and using the fact that $e^{-x^2}$ is even in variable $x \in \mathbb R$, we conclude that
\begin{flalign*}
	\int_{B_r^4 \cap \varPi_4}\mathcal G(X)\, dX &=\frac{c r^2 }{2(c-1)\sqrt{ \pi}} \int_0^{r\,\sqrt{c/(8(c-1))}} e^{-\xi^2} d\xi
	-\frac{ r^2 }{2\sqrt{\pi}} \int_0^{r/\sqrt{8}}e^{-\eta^2} d\eta\\
	&=\frac{c r^2 }{4(c-1)\sqrt{ \pi}} \int_{-r\,\sqrt{c/(8(c-1))}}^{r\,\sqrt{c/(8(c-1))}} e^{-\xi^2} d\xi
	-\frac{ r^2 }{4\sqrt{\pi}} \int_{-r/\sqrt{8}}^{r/\sqrt{8}}e^{-\eta^2} d\eta\\
	&=\frac{c r^2 }{4(c-1)} \erf \left(\sqrt{ \frac{c r^2}{8(c-1)} }  \right) -\frac{ r^2 }{4} \erf \left(\sqrt{ \frac{ r^2}{8} }   \right).
\end{flalign*}

Since in this case $r=\sqrt{12(c-1)\log(\frac{c}{c-1})}$, when $d=2$ Theorem~\ref{main} can be enunciated as follows:
\begin{theorem}\label{theo2}
	If $d=2$ then for all $c\geq 4$,
	\begin{flalign*}
		d_n(cn)&= 3 c \log\left(\frac{c}{c-1}\right) \erf \left(\sqrt{ \frac{3c \log(c/(c-1))}{2} }  \right)  \\
		&\hspace{1cm} -3(c-1)\log\left(\frac{c}{c-1}\right)  \erf \left(\sqrt{ \frac{ 3(c-1)\log(\frac{c}{c-1})}{2} }   \right) + \mathcal{O}\left(\frac{1}{\sqrt{n}}\right).
	\end{flalign*}
\end{theorem} 

Figure~\ref{graph2} shows the graph of the function given in Theorem~\ref{theo2} which is the asymptotic profile of $d_n(cn)$ when $d=2$. 
\begin{figure}[H]
	\centering
	\includegraphics[scale=0.35]{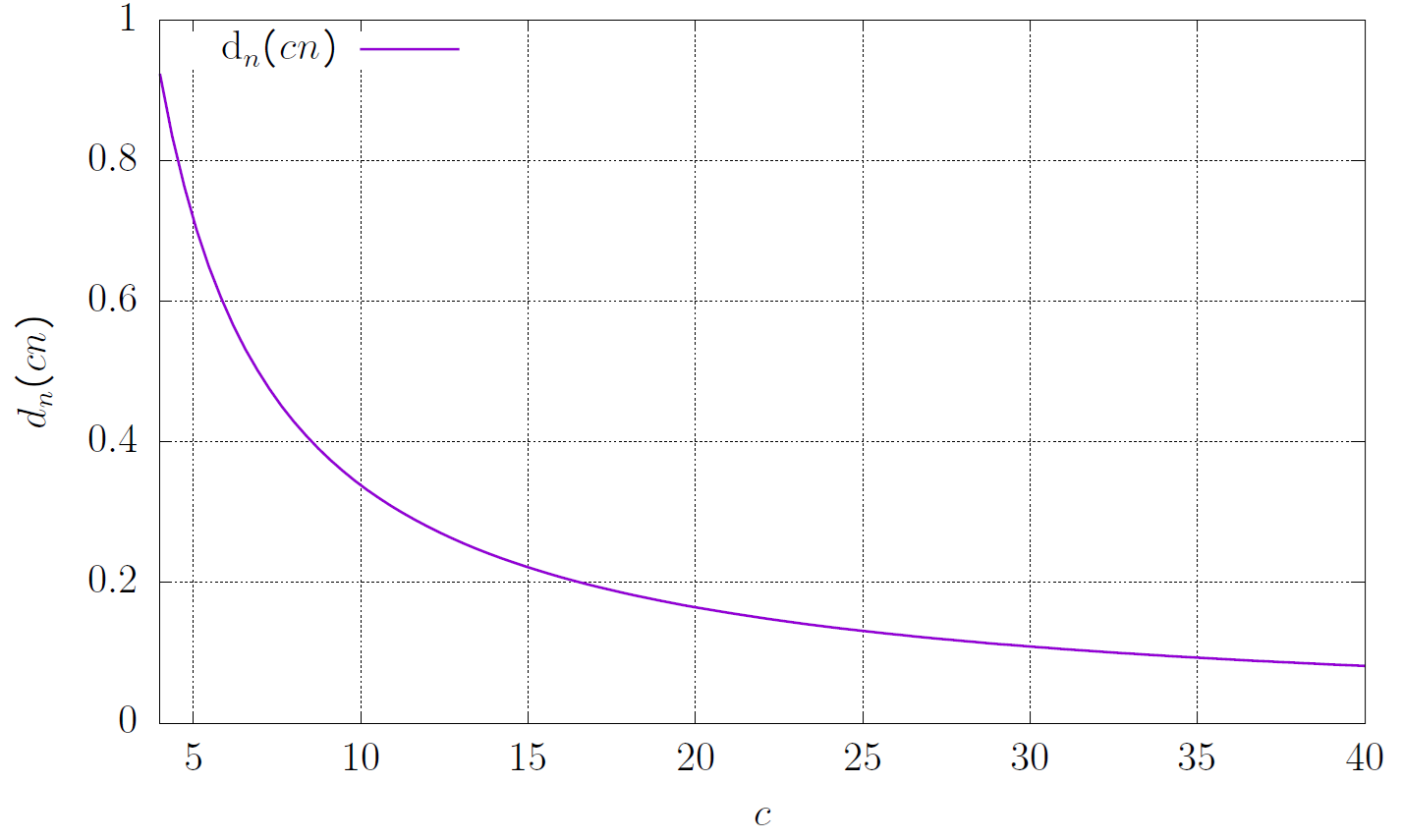}
	\caption{Graph of the asymptotic profile of $d_n(cn)$ when $d=2$.}
	\label{graph2}
\end{figure}
We can conclude that using a typical $52$-card deck is not very good to simulate the SSRW on $\mathbb Z^2$ as it is to simulate the SSRW on $\mathbb Z$. In that case, one must use a deck with more cards.

\begin{remark}[Using a deck with more suits]
	There are many different ways to calculate $\int_{B_r^{2d} \cap \varPi_{2d}}\mathcal G(X)\, dX$. Although the procedure used to obtain the change of variables (\ref{system}) is very exhausting, it can be used for any $d \in \mathbb N$ and so can the spherical coordinates (see~\cite{bartle}). This explains why we decided to analyze the bi-dimensional problem in this way. 
\end{remark}

\section{Proof of the main result}\label{proof}

For $(\lambda_1,\ldots,\lambda_{2d})\in\Pi_{2d}^n$, and recalling that $K=N/2d$, we introduce the notation
\begin{flalign}\label{f 1}
	f(\lambda_1,\lambda_2,\ldots,\lambda_{2d},N,n):&= (2d)^{n} \frac{\binom{N-n}{K-\lambda_1,\ldots,K-\lambda_{2d}}}{\binom{N}{K,\ldots, K}}.
\end{flalign}  
Notice that this function represents the Radon-Nikodym derivative of $H_n$ with respect to $B_n$. We can now rewrite equation \eqref{dvt2}, for $N=cn$, as \begin{equation}\label{dvt3}
	d_n(cn)=\sum_{\lambda_1=0}^n \sum_{\lambda_2=0}^{n-\lambda_1}\cdots \sum_{\lambda_{2d}=0}^{n-\sum_{i=1}^{2d-1} \lambda_i} \frac{1}{(2d)^n}
	\binom{ n}{\lambda_1, \ldots, \lambda_{2d}}
	\left[	f(\lambda_1,\lambda_2,\ldots,\lambda_{2d},cn,n)-1\right]^+.
\end{equation}

The positivity that appears in the parcels of \eqref{dvt3} is inconvenient for our task of estimating $d_n(cn)$. The first step of this proof is to get rid of it. This is done by localizing the vectors in
\begin{flalign}\label{zero1}
	\Delta_n= \left\{ \lambda=(\lambda_1,\ldots,\lambda_{2d}) \in \Pi_{2d}^n; \,  f(\lambda_1,\lambda_2,\ldots,\lambda_{2d},cn,n)\geq 1  \right\}.
\end{flalign}

The next lemma, which is more of an observation, states a monotonicity property satisfied by the function $f$.

\begin{lemma}\label{monotonia}
	Let $\lambda=(\lambda_1,\ldots,\lambda_{2d})\in\Pi_{2d}^n$. For $\lambda_i<\lambda_j$ let us define $\lambda^{i,j}=(\lambda^{i,j}_1,\ldots,\lambda^{i,j}_{2d})$ as
	\begin{flalign}\label{lambda_modificado}
		\lambda^{i,j}_\ell=\begin{cases}
			\lambda_i +1, &\text{~~if~~} \ell=i\\
			\lambda_j -1, &\text{~~if~~}\ell=j\\
			\lambda_\ell, &\text{~~if~~}\ell\notin\{i,j\}.
		\end{cases}
	\end{flalign}
	Then 
	\begin{flalign*}
		f(\lambda, N, n)\leq f(\lambda^{i,j}, N, n).
	\end{flalign*}
\end{lemma}
\begin{proof}
	We only need to prove that
	\begin{flalign*}
		\binom{N-n}{K-\lambda_1,\ldots,K-\lambda_{2d}}\leq \binom{N-n}{K-\lambda^{i,j}_1,\ldots,K-\lambda^{i,j}_{2d}},
	\end{flalign*}
	which is equivalent to
	\begin{flalign*}
		(K-\lambda^{i,j}_i)!(K-\lambda^{i,j}_j)!\leq (K-\lambda_i)!(K-\lambda_j)!\,.
	\end{flalign*}
	This holds if and only if $\lambda_i+1\leq\lambda_j$, and this completes the proof.
\end{proof}

We will need some technical lemmas in order to find the location of the points in $\Delta_n$. We start by the classical Stirling's approximation (see for instance~\cite{franco}):
\begin{equation}\label{stirling}
	m!=\sqrt{2\pi m}\left(\frac{m}{e}\right)^m(1+\mathcal O(1/m)).
\end{equation}
As a direct consequence of the Stirling's approximation, we also have
\begin{flalign}\label{conseqStirling}
	\binom{m}{k,\ldots,k}=\frac{\sqrt{2} \, d^d \,(2d)^m}{\pi^{d-1/2}\, m^{d-1/2}} \, (1+\mathcal O(1/m))
\end{flalign}
for $k=m/2d$.

Recall the definition of the euclidean norm given in (\ref{euclidean}). Denote the $\ell_\infty$-norm of a vector $v=(v_1,\ldots,v_{m}) \in \mathbb R^m$ by $\| v \|_{\infty}=\max_{1 \leq i \leq m} |v_i|$. The next lemma gives us an asymptotic expression for $\binom{m}{k_1,\ldots,k_{2d}}$ when the $\ell_\infty$-norm of $(k_1,\ldots,k_{2d}) \in \Pi_{2d}^m$ is not too far from $ (\frac{m}{2d},\ldots,\frac{m}{2d})$. The result is standard but we will present the proof for the seek of completeness.

\begin{lemma}\label{maumau}
	Let $L_m=(\ell_m^{1},\ldots,\ell_m^{2d}) \in \mathbb R^{2d}$ define a sequence satisfying $\sum_{i=1}^{2d} \ell_m^i = 0$ and $\lim_{m\to\infty}\frac{\|L_m\|_{\infty}}{m^{2/3}}=0$. If $k_i=(m+\ell_m^i)/(2d)$ for all $i \in \{ 1,\ldots,2d \}$ then
	\begin{equation*}
		\binom{m}{k_1,\ldots,k_{2d}}=\frac{\sqrt{2} \,d^d \, (2d)^m }{\pi^{d-1/2}\,m^{d-1/2}} \, \exp{\left\{  
			- \frac{\| L_m \|^2}{4 d m}
			\right\}}\left( 1+\mathcal O \left(\frac{\| L_m\|_{\infty}^3}{m^2}\right)  \right).
	\end{equation*}
\end{lemma}
\begin{proof}
	For each $i \in \{1,\ldots,2d\}$ define 
	\begin{flalign*}
		\varepsilon_i=\frac{\ell_m^i}{m},
	\end{flalign*}
	and notice that, since $\ell_m^i/m^{2/3}\to 0$ as $m\to\infty$, we have, in particular, that $\varepsilon_i\to 0$ as $m\to\infty$. Thus, we can obtain the expression $k_i=\frac{m}{2d}(1+\varepsilon_i)$. Therefore, by Stirling's formula, we have
	\begin{flalign*}
		k_i!&=(2\pi)^{1/2} \left(\frac{m}{2d}\right)^{1/2}  \left(\frac{m}{2d e}\right)^{k_i} (1+\varepsilon_i)^{k_i} \left( 1+\frac{\ell_m^i}{m} \right)^{1/2} \left( 1+\mathcal O \left(\frac{1}{m}\right)  \right)\\
		&=(2\pi)^{1/2} \left(\frac{m}{2d}\right)^{1/2}  \left(\frac{m}{2d e}\right)^{k_i} (1+\varepsilon_i)^{k_i}\left( 1+\mathcal O \left(\frac{\| L_m\|_{\infty}}{m}\right)  \right).
	\end{flalign*}
	Since $\sum_{i=1}^{2d} k_i=m$, we have
	\begin{flalign*}
		\prod_{i=1}^{2d} k_i! =(2\pi)^{d} \left(\frac{m}{2d}\right)^{d}  \left(\frac{m}{2d e}\right)^{m} \left(  \prod_{i=1}^{2d} (1+\varepsilon_i)^{ k_i} \right) \left( 1+\mathcal O \left(\frac{\| L_m\|_{\infty}}{m}\right)  \right).
	\end{flalign*}
	Hence, by (\ref{stirling}) we have
	\begin{flalign}\label{quasela}
		\binom{m}{k_i,\ldots,k_{2d}}&= \frac{m!}{\prod_{i=1}^{2d}k_i}  \nonumber\\
		&= \frac{\sqrt{2}\, d^d\, (2d)^m}{\pi^{d-1/2}\,m^{d-1/2}}\,\left(  \prod_{i=1}^{2d} (1+\varepsilon_i)^{ k_i} \right)^{-1}\,  \left( 1+\mathcal O \left(\frac{\| L_m\|_{\infty}}{m}\right)  \right).
	\end{flalign}
	Now observe that
	\begin{flalign*}
		\prod_{i=1}^{2d} (1+\varepsilon_i)^{ k_i}&=\exp{\left\{ \sum_{i=1}^{2d} k_i \,\log (1+\varepsilon_i)    \right\}}\\
		& = \exp{\left\{ \frac{m}{2d} \sum_{i=1}^{2d} (1+\varepsilon_i) \,\log (1+\varepsilon_i)    \right\}}.
	\end{flalign*}
	Since $(1+\varepsilon) \log(1+\varepsilon) = \varepsilon +\frac{\varepsilon^2}{2} + \mathcal O (|\varepsilon|^3)$ when $\varepsilon \to 0$ and $\sum_{i=1}^{2d} \varepsilon_i =0$, we have
	\begin{flalign}
		\prod_{i=1}^{2d} (1+\varepsilon_i)^{ k_i}
		& = \exp{\left\{ \left( \frac{m}{2d}  \sum_{i=1}^{2d} \varepsilon_i \right) + \left( \frac{m}{4d}  \sum_{i=1}^{2d} |\varepsilon_i|^2 \right) + \mathcal O \left(\frac{\| L_m\|_{\infty}^3}{m^2}\right)  \right\}} \nonumber \\
		&=\exp{\left\{   \left( \frac{m}{4d}  \sum_{i=1}^{2d} \left| \frac{\ell_m^i}{m} \right|^2 \right) + \mathcal O \left(\frac{\| L_m\|_{\infty}^3}{m^2}\right)  \right\}} \nonumber \\
		&=\exp{\left\{   \left( \frac{1}{4dm}  \sum_{i=1}^{2d} |\ell_m^i|^2 \right) + \mathcal O \left(\frac{\| L_m\|_{\infty}^3}{m^2}\right)  \right\}}  \nonumber  \\
		&=\exp{\left\{   \left( \frac{1}{4dm}  \sum_{i=1}^{2d} |\ell_m^i|^2 \right) \right\}} \,\left(1 + \mathcal O \left(\frac{\| L_m\|_{\infty}^3}{m^2}\right) \right). \nonumber 
	\end{flalign}
	The proof is completed putting the above estimate in~(\ref{quasela}) and noticing that, since $\|L_m\|_{\infty}/m^{2/3}\to 0$, we can replace $(1+\mathcal O(\| L_m\|_{\infty}^3/m^2))^{-1}$ by $(1+\mathcal O(\| L_m\|_{\infty}^3/m^2))$.
\end{proof}

Notice that the obtained asymptotic expression suggests that the parametrization $L_m=~\sqrt{m}\, A_m$, $A_m=(a_m^1,\ldots,a_m^{2d}) \in \mathbb R^{2d}$, could be useful. That way we get
\begin{equation}\label{maumaufinal}
	\binom{m}{ \frac{m + a_m^1 \sqrt{m}}{2d},\ldots,\frac{m + a_m^{2d} \sqrt{m}}{2d}}=\frac{\sqrt{2} \,d^d \, (2d)^m }{\pi^{d-1/2}\,m^{d-1/2}} \, \exp{\left\{  
		- \frac{\| A_m \|^2}{4 d }
		\right\}}\left( 1+\mathcal O \left(\frac{\| A_m\|_{\infty}^3}{\sqrt{m}}\right)  \right),
\end{equation}
when $\|A_m\|_{\infty}/m^{1/6}\to 0$.

We are now in condition to return to the task of characterizing the vectors in $\Delta_n$.

\begin{lemma}\label{asymtotic_mean}
	If $A_n=(a_n^1,\ldots,a_n^{2d}) \in \mathbb R^{2d}$ defines a sequence satisfying $\sum_{i=1}^{2d} a_n^i =0$ and $\lim_{n\to\infty}\|A_n\|_\infty/n^{1/6}=0$, then 
	\begin{equation*}
		f\!\left(\frac{n + a_n^1\sqrt{n}}{2d},\ldots,\frac{n + a_n^{2d}\sqrt{n}}{2d},cn,n\right)\!=\!\left(\frac{c}{c-1}\right)^{d-\frac{1}{2}}\!\exp\!{\left\{-\frac{\| A_n \|^2}{4d(c-1)}\right\}}\!\left(1+\mathcal O\left(\frac{\|A_n\|_{\infty}^3}{\sqrt{n}}\right)\right).
	\end{equation*}
\end{lemma}

\begin{proof}
	Taking $N=cn$ and $\lambda_i=\frac{n + a_n^i\sqrt{n}}{2d}$ for every $i \in \{1,\ldots,2d\}$, we have
	\begin{flalign}\label{eq1}
		\binom{N-n}{K-\lambda_1,\ldots,K-\lambda_{2d}}=	\binom{(c-1)n}{\frac{(c-1)n - ({a_n^1}/{\sqrt{c-1}})\sqrt{(c-1)n}}{2d},\ldots,\frac{(c-1)n - ({a_n^{2d}}/{\sqrt{c-1}})\sqrt{(c-1)n}}{2d}}
	\end{flalign}
	and
	\begin{flalign}\label{eq2}
		\binom{N}{K,\ldots, K}=	\binom{cn}{cn/(2d),\ldots, cn/(2d)}.
	\end{flalign}
	By identity (\ref{eq1}), and using (\ref{maumaufinal}) for $m=N-n=(c-1)n$ we obtain
	\begin{flalign*}
		\binom{N-n}{K-\lambda_1,\ldots,K-\lambda_{2d}}=\frac{\sqrt{2} \,d^d \, (2d)^{(c-1)n} }{\pi^{d-1/2}\,{[(c-1)n]}^{d-1/2}} \exp{\left\{  
			- \frac{\| A_n \|^2}{4 d (c-1) }
			\right\}}\left( 1+\mathcal O \left(\frac{\| A_n\|_{\infty}^3}{\sqrt{n}}\right)  \right).
	\end{flalign*}
	By identities (\ref{conseqStirling}) and (\ref{eq2}) we have
	\begin{flalign*}
		\binom{N}{K,\ldots, K}=	\frac{\sqrt{2} \, d^d \,(2d)^{cn}}{\pi^{d-1/2}\, {(cn)}^{d-1/2}} \, (1+\mathcal O(1/n)).
	\end{flalign*}
	The result follows from the definition of $f$ given in (\ref{f 1}). 
\end{proof}

Recall the definition of $\Delta_n$ given in (\ref{zero1}). Let us define 
\begin{equation*}
	\Phi_n=\left\{A_n=(a_n^1,\ldots,a_n^{2d}) \in\mathbb R^{2d}; \left(\frac{n+a_n^1\sqrt{n}}{2d},\ldots,\frac{n+a_n^{2d}\sqrt{n}}{2d}\right)\in\Delta_n\right\}.
\end{equation*}
Notice that
\begin{flalign*}
	\sum_{i=1}^{2d} a_n^i = 0 \,\, \text{ for every } A_n=(a_n^1,\ldots,a_n^{2d}) \in \Phi_n
\end{flalign*}
Observe that defining $\psi: \Delta_n \to \Phi_n$ by
\begin{flalign}\label{bijection}
	\psi(\lambda_1,\ldots,\lambda_{2d})=\left(  \frac{2d \lambda_1}{\sqrt{n}} - \sqrt{n},\ldots,   \frac{2d \lambda_{2d}}{\sqrt{n}} - \sqrt{n} \right),
\end{flalign}
we obtain a bijection between $\Delta_n$ and $\Phi_n$. Therefore, characterizing $\Delta_n$ and $\Phi_n$ are equivalent tasks. In order to accomplish these tasks, we prove the following:

\begin{lemma}\label{gamma_bounded}
	The set $\bigcup\limits_{n=1}^\infty\Phi_n$ is bounded. 
\end{lemma}

\begin{proof}
	Notice that, for each $n$, the set $\Phi_n$ is finite. We will prove that any sequence defined by
	$A_n=(a_n^1,\ldots,a_n^{2d})$, such that $A_n\in\Phi_n$ for all $n$, is bounded. If $(A_n)_n$ were unbounded, then there would be a subsequence $(A_{n_k})_k$ such that $\|A_{n_k}\|_\infty\to\infty$ as $k \to \infty$.
	
	We will split our analysis in two cases: $\|A_{n_k}\|_{\infty}/n_k^{1/6}\rightarrow 0$ or $\|A_{n_k}\|_{\infty}/n_k^{1/6}\nrightarrow 0$ as $k\to\infty$.
	
	In the first case we use Lemma \ref{asymtotic_mean} to conclude that
	\begin{equation*}
		\lim_{k\to\infty} f\left(\frac{n_k + a_{n_k}^1\sqrt{n_k}}{2d},\ldots,\frac{n_k + a_{n_k}^{2d}\sqrt{n_k}}{2d},cn_k,n_k\right)=0, 
	\end{equation*}
	which contradicts the fact that of $A_{n_k} \in \Phi_{n_k}$ for all $k$.
	
	In the second case, there would exist $\varepsilon>0$ and a subsequence of $(A_{n_k})_k$, which, for simplicity, we will still call $(A_{n_k})_k$, such that $\|A_{n_k}\|_{\infty}>\varepsilon n_k^{1/6}$ for all $k$. 
	
	Now we will modify the vector $\lambda_{n_k}=(\lambda^1_{n_k},\ldots,\lambda^{2d}_{n_k})\in\Pi^{n_k}_{2d}$, where $\lambda^\ell_{n_k}=\frac{n_k+a_{n_k}^\ell\sqrt{n_k}}{2d}$ in the following way: consider $\lambda_{n_k}^i$ and $\lambda_{n_k}^j$ two coordinates that assume, respectively, the smallest and the biggest value among all coordinates $\lambda_{n_k}^\ell$, then we replace the vector $\lambda_{n_k}$ by the vector $\lambda_{n_k}^{i,j}$ defined in \eqref{lambda_modificado}. By Lemma \ref{monotonia} we will have 
	\begin{flalign*}
		f(\lambda_{n_k},cn_k,n_k)\leq f(\lambda_{n_k}^{i,j},cn_k,n_k). 
	\end{flalign*}
	
	We repeatedly execute the same procedure described above, always considering at each step the smallest and the biggest of all current coordinates, until we reach a vector $\tilde\lambda_{n_k}=(\tilde\lambda_{n_k}^1\ldots,\tilde\lambda_{n_k}^{2d})$ that is close enough to the vector $(\frac{n_k}{2d},\ldots,\frac{n_k}{2d})$ in the $\ell_\infty$-norm. More specifically, for a fixed $0<\delta<1/6$, we will ask that $n_{k}^{1/6-\delta}\leq \|\tilde A_{n_k}\|_\infty\leq 2n_k^{1/6-\delta}$, where $\tilde A_{n_k}=(\tilde a_{n_k}^1,\ldots,\tilde a_{n_k}^{2d})$ for $\tilde a_{n_k}^\ell$ defined by $\tilde\lambda_{n_k}^\ell=\frac{n_k+\tilde a_{n_k}^\ell}{2d}$.   
	Therefore, again by Lemma \ref{monotonia}, 
	\begin{flalign}\label{monotony_f}
		f(\lambda_{n_k},cn_k,n_k)\leq f(\tilde\lambda_{n_k},cn_k,n_k).
	\end{flalign}
	By Lemma \ref{asymtotic_mean}, we have
	
	\begin{flalign*}
		f(\tilde\lambda_{n_k},cn_k,n_k)
		&=\left(\frac{c}{c-1}\right)^{d-1/2}\,\exp{\left(-\frac{\| \tilde A_{n_k} \|^2}{4d(c-1)}\right)}\left(1+\mathcal O\left(\frac{\|\tilde A_{n_k}\|_{\infty}^3}{\sqrt{n}}\right)\right)\\
		&\leq \left(\frac{c}{c-1}\right)^{d-1/2}\,\exp{\left(-\frac{n_k^{1/3-2\delta}}{4d(c-1)}\right)}\left(1+\mathcal{O}\left(\frac{1}{n_k^{3\delta}}\right)\right).
	\end{flalign*}
	The last line above goes to zero as $k\to\infty$. By \eqref{monotony_f} we have that $f(\lambda_{n_k},cn_k,n_k)$ will go to zero as well, and this again contradicts the fact that of $A_{n_k} \in \Phi_{n_k}$ for all $k$.
\end{proof}

Let $B_r^m=\left\{  X \in \mathbb R^{m}; \| X \| \leq r  \right\}$ be the $m$-ball that is centered at the origin and whose radius is equal to $r$. Let $\varPi_{m}=\{ x=(x_1,\ldots,x_{m}) \in \mathbb R^{m} ; \sum_{i=1}^{m} x_i=0\}$ be the hyperplane that is orthogonal to the vector $(1,\ldots,1)$ and contains the origin. The next proposition gives us a precise characterization of the set $\Phi_n$ showing that it is a partition of the compact set $B_r^{2d}\cap \varPi_{2d}$ where $r$ is the constant defined in Theorem \ref{main}, being the volume of each partition cube equal to $(2d/\sqrt{n})^{2d-1}$. 

\begin{proposition}\label{muda_sinal}
	Let $r=\sqrt{(4d^2-2d)(c-1)\log(\frac{c}{c-1})}$. The set $\Phi_n$  is a partition of the compact set $B_r^{2d}\cap \varPi_{2d}$ for which the volume of each partition cube is equal to $(2d/\sqrt{n})^{2d-1}$. 
\end{proposition}
\begin{proof}
	Let $A_n=(a_n^1,\ldots,a_n^{2d})$ be a sequence in $\Phi_n$. 
	From the definition of $\Phi_n$ and by Lemma \ref{asymtotic_mean} we have
	\begin{flalign*}
		1&\leq f\left(\frac{n + a_n^1\sqrt{n}}{2d},\ldots,\frac{n + a_n^{2d}\sqrt{n}}{2d},cn,n\right)\\
		&=\left(\frac{c}{c-1}\right)^{d-1/2}\,\exp{\left(-\frac{\| A_n \|^2}{4d(c-1)}\right)}(1+\mathcal O(\|A_n\|_{\infty}^3/\sqrt{n})).
	\end{flalign*}
	By Lemma \ref{gamma_bounded}, we obtain
	\begin{flalign*}
		1&\leq \left(\frac{c}{c-1}\right)^{d-1/2}\,\exp{\left(-\frac{\| A_n \|^2}{4d(c-1)}\right)}(1+\mathcal O(1/\sqrt{n})).
	\end{flalign*}
	Isolating the term $\|A_n\|$ above, we see that $\|A_n\|\leq r+\mathcal O(1/\sqrt{n})$. Moreover, since the euclidean distance between two adjacent points of the set $\Phi_n$ is always equal to $2d/\sqrt{n}$, we see that $\Phi_n$ is a partition of the set $B_r^{2d} \cap \varPi_{2d}\subset \mathbb R^{2d-1}$ for which the volume of each of the partition cubes is equal to $(2d/\sqrt{n})^{2d-1}$.
\end{proof}

\begin{proof}[Proof of Theorem \ref{main}]
	By \eqref{dvt3}, and since $f(\lambda_n^1,\ldots,\lambda_n^{2d},cn,n)-1$ is non negative when $\big(\lambda_n^1,\ldots,\lambda_n^{2d}\big)\in\Delta_n$, we can rewrite the total variation distance \eqref{dvt3} as
	\begin{flalign*}
		d_n(cn)=\sum_{(\lambda_n^1,\ldots,\lambda_n^{2d}) \in \Delta_n } \frac{1}{(2d)^n}
		\binom{ n}{\lambda_n^1, \ldots, \lambda_n^{2d}}
		\big(	f(\lambda_n^1,\lambda_n^2,\ldots,\lambda_n^{2d},cn,n)-1\big).
	\end{flalign*} 
	Recall the definition of $\psi:\Delta_n \to \Phi_n$ given in (\ref{bijection}). Thus we can rewrite the right-hand side above as
	\begin{flalign*}
		\sum_{(a_n^1,\ldots,a_n^{2d}) \in \Phi_n } \frac{1}{(2d)^n}
		\binom{ n}{\frac{n+a_n^1 \sqrt{n}}{2d}, \ldots, \frac{n+a_n^{2d} \sqrt{n}}{2d}}
		\left(	f\left(\frac{n+a_n^1 \sqrt{n}}{2d},\ldots,\frac{n+a_n^{2d} \sqrt{n}}{2d},cn,n\right)-1\right).
	\end{flalign*}
	By Lemma \ref{asymtotic_mean}, we have
	\begin{align*}
		d_n(cn)&=\sum_{A_n=(a_n^1,\ldots,a_n^{2d}) \in \Phi_n } \frac{1}{(2d)^n}
		\binom{ n}{\frac{n+a_n^1 \sqrt{n}}{2d}, \ldots, \frac{n+a_n^{2d} \sqrt{n}}{2d}}\\
		& \hspace{4cm}\times
		\left[\left(\frac{c}{c-1}\right)^{d-1/2}\exp{\left(-\frac{\| A_n \|^2}{4d(c-1)}\right)}-1+\mathcal{O}\left(\frac{1}{\sqrt{n}}\right)\right]\\
		&=\sum_{A_n=(a_n^1,\ldots,a_n^{2d}) \in \Phi_n } \frac{1}{(2d)^n}
		\binom{ n}{\frac{n+a_n^1 \sqrt{n}}{2d}, \ldots, \frac{n+a_n^{2d} \sqrt{n}}{2d}}\\
		& \hspace{4cm}\times
		\left[\left(\frac{c}{c-1}\right)^{d-1/2}\exp{\left(-\frac{\| A_n \|^2}{4d(c-1)}\right)}-1\right]+\mathcal{O}\left(\frac{1}{\sqrt{n}}\right),
	\end{align*}
	where in the last line above it was used that
	\begin{flalign*}
		\sum_{A_n=(a_n^1,\ldots,a_n^{2d}) \in \Phi_n } \frac{1}{(2d)^n}
		\binom{ n}{\frac{n+a_n^1 \sqrt{n}}{2d}, \ldots, \frac{n+a_n^{2d} \sqrt{n}}{2d}}=\sum_{(\lambda_n^1,\ldots,\lambda_n^{2d}) \in \Delta_n } \frac{1}{(2d)^n}
		\binom{ n}{\lambda_n^1, \ldots, \lambda_n^{2d}}\leq 1.
	\end{flalign*}
	Also, by \eqref{maumaufinal}, for $A_n=(a_n^1,\ldots,a_n^{2d})\in\Phi_n$ we have
	\begin{equation*}
		\binom{ n}{\frac{n+a_n^1 \sqrt{n}}{2d}, \ldots, \frac{n+a_n^{2d} \sqrt{n}}{2d}}=
		\frac{\sqrt{2} \,d^d \, (2d)^n }{\pi^{d-1/2}\,n^{d-1/2}} \, \exp{\left\{  
			- \frac{\| A_n \|^2}{4 d }
			\right\}}
		(1+\mathcal O(1/\sqrt{n})).
	\end{equation*}
	Then, 
	\begin{flalign*}
		d_n(cn)&=\sum_{A_n\in\Phi_n}\left( \pi^{1/2-d} \, 2^{3/2-2d}\,d^{1-d} \right) \left(\frac{2d}{\sqrt{n}}\right)^{2d-1}\\
		& \hspace{2cm}\times \left[\left(\frac{c}{c-1}\right)^{d-1/2}\exp{\left(-\frac{c\,\|A_n\|^2}{4d(c-1)}\right)}-\exp{\left(-\frac{\|A_n\|^2}{4d}\right)}\right]+\mathcal{O}\left(\frac{1}{\sqrt{n}}\right).
	\end{flalign*}
	Notice, by Proposition \ref{muda_sinal}, that the right hand side above represents a Riemann sum of the Gaussian profile $\mathcal G: B_r^{2d} \cap \varPi_{2d} \to [0,1]$ defined as
	\begin{equation*}
		\mathcal G(X)=\left( \pi^{1/2-d} \, 2^{3/2-2d}\,d^{1-d} \right) \left[\left(\frac{c}{c-1}\right)^{d-1/2}\exp{\left(-\frac{c\|(X)\|^2}{4d(c-1)}\right)}-\exp{\left(-\frac{\|X\|^2}{4d}\right)}\right].
	\end{equation*}
	The order of the error in the approximation of the Riemann sum by the integral is equal to the volume of each of the partition cubes, which is equal to $(2d/\sqrt{n})^{2d-1}$. Hence, we have
	\begin{equation*}
		d_n(cn)=\int_{B_r^{2d} \cap \varPi_{2d}} \mathcal G(X)\, dX+\mathcal{O}\left(\frac{1}{\sqrt{n}}\right).
	\end{equation*}
	
\end{proof}

\vspace{1cm}

\textbf{Acknowledgments:}
Part of this work was carried out during the undergraduate research program "Jornadas de Pesquisa em Matemática do ICMC 2023" held at the Instituto de Ciências Matemáticas e de Computação (ICMC) - Universidade de São Paulo (USP), and which was partially supported by the Centro de Ciências Matemáticas Aplicadas à Indústria (CeMEAI - CEPID) under FAPESP Grant $\#$ 2013/07375-0, and by the FAPESP Grant $\#$ 2019/16061-2. We thank the hospitality of ICMC-USP during the program.




\end{document}